\documentclass[11pt,a4paper]{article}

\usepackage[T1]{fontenc}
\usepackage[utf8]{inputenc}
\usepackage{amsmath,amssymb,amsthm,mathtools}
\usepackage{geometry}
\usepackage{booktabs}
\usepackage{array}
\usepackage{longtable}
\usepackage{microtype}
\usepackage{hyperref}
\usepackage{authblk}
\usepackage{xcolor}
\hypersetup{
  colorlinks=true,
  linkcolor=blue!55!black,
  citecolor=blue!55!black,
  urlcolor=blue!55!black
}

\newcommand{\C}{\mathbb{C}}
\newcommand{\R}{\mathbb{R}}
\newcommand{\Z}{\mathbb{Z}}
\newcommand{\QLS}{\mathrm{QLS}}
\newcommand{\card}{\operatorname{card}}
\newcommand{\od}{\mathbin{\odot}}
\newcommand{\ii}{\mathrm{i}}
\newcommand{\e}{\mathrm{e}}
\newcommand{\ip}[2]{\left\langle #1,#2\right\rangle}

\newtheorem{definition}{Definition}

\newtheorem{proposition}{Proposition}
\newtheorem{theorem}{Theorem}
\newtheorem{corollary}{Corollary}
\newtheorem{remark}{Remark}

\title{New Cardinalities for Quantum Latin Squares of Order Six}

\author[1]{Zhipeng Xu\thanks{Corresponding author: \texttt{xuzhp@ntu.edu.cn}}}
\affil[1]{School of Mathematics and Statistics, Nantong University, Nantong, China}
\date{}

\begin{document}
\maketitle

\begin{abstract}
We give explicit quantum Latin squares of order $6$ with cardinalities
$19$, $21$, $23$, $25$, $27$, $32$, and $35$, where vectors differing only
by a global phase are identified.  Cardinalities $19$ and $21$ arise from
symmetric Schur products of columns of dephased Butson matrices.  A
parameterized direct-sum construction in $\C^6=\C^4\oplus\C^2$ yields
cardinalities $23$ and $25$ by a controlled splitting of two pairs of rays.
Cardinality $27$ is obtained from mixed Schur products of a $BH(6,6)$ matrix
and a row-permuted copy.  Two further mixed constructions give cardinalities
$32$ and $35$: the first has exactly two three-element phase classes, while
the second has exactly one two-element phase class.  In every Butson case,
Hadamard orthogonality and the ray count are certified by finite arithmetic
with exponent matrices.  Combined with previously known attainable values
and the general impossibility of cardinality $7$, these constructions realize
every order-six cardinality except $7$ and the single currently unresolved
value $29$.
\end{abstract}


\section{Introduction}

Latin squares are classical objects of combinatorics and design theory whose
history includes Euler's study of Graeco--Latin squares
\cite{Euler1782,DenesKeedwell1974,ColbournDinitz2007}.  Quantum Latin squares
were introduced by Musto and Vicary as a quantum analogue in which the symbols
of a Latin square are replaced by unit vectors and every row and column is an
orthonormal basis \cite{MustoVicary2016}.  They are connected with unitary
error bases, teleportation and dense coding, and more general quantum
combinatorial designs \cite{Werner2001,GoyenecheEtAl2018,MustoVicary2019}.

Order six is especially prominent because of Euler's classical 36-officers
problem and its quantum counterparts \cite{RatherEtAl2022,ZyczkowskiEtAl2023}.
A separate line of work studies the \emph{cardinality} of a quantum Latin
square, namely the number of distinct vectors after global phases have been
identified \cite{PaczosEtAl2021,ZhangWangJi2025,ZangEtAl2025,ZhangCao2026,
ZhangJi2026}.  For a $\QLS(n)$, the cardinality lies between $n$ and $n^2$,
and cardinality $n+1$ is impossible \cite{ZangEtAl2025,ZhangJi2026}.

A summary of known cardinalities is given by Zhang, Lv, and Cao
\cite[Table~6]{ZhangLvCao2026}.  In that account, the order-six cases
$23$, $25$, $27$, $29$, $32$, and $35$ were left open.  The present paper
constructs five of these six values, leaving only $29$ unresolved.  Together
with the previously known cases, this realizes every cardinality
$c\in[6,36]\setminus\{7,29\}$; the value $7$ is impossible for every
$\QLS(6)$.

The seven constructions display three complementary mechanisms.  First,
symmetric Schur products of the columns of one complex Hadamard matrix
produce at most $\binom72=21$ phase classes.  A designed triple collision in a
$BH(6,8)$ matrix gives cardinality $19$, whereas Tao's isolated $BH(6,3)$
matrix reaches the collision-free maximum $21$.  Second, a linked family of
orthonormal bases in $\C^4\oplus\C^2$ yields cardinality $23$ at a symmetric
parameter point and cardinality $25$ after a two-pair ray splitting.  Third,
mixed Schur products of two Butson matrices remove the symmetry bound.  They
produce cardinality $27$ from a row-permuted $BH(6,6)$ pair, cardinality $32$
from a $BH(6,3)$--$BH(6,6)$ pair, and cardinality $35$ from a $BH(6,12)$
matrix and a row-permuted transpose.

The cardinality arguments use modular exponent signatures for the Butson
constructions and exact inner-product identities for the direct-sum family.
Thus none of the existence proofs depends on numerical tolerance.

The paper is organized as follows.  Section~\ref{sec:prelim} recalls the
necessary definitions.  Section~\ref{sec:hadprod} gives the mixed and symmetric
Hadamard-product constructions.  Sections~\ref{sec:c19} and \ref{sec:c21}
give the cardinality-$19$ and cardinality-$21$ examples.  Section~\ref{sec:direct}
gives a common direct-sum family for cardinalities $23$ and $25$.
Sections~\ref{sec:c27}, \ref{sec:c32}, and \ref{sec:c35} give mixed Butson
constructions of cardinalities $27$, $32$, and $35$, respectively.
Section~\ref{sec:spectrum} records the resulting spectrum consequences.

\section{Definitions and background}\label{sec:prelim}

\begin{definition}[Quantum Latin square]
A quantum Latin square of order $n$, denoted by $\QLS(n)$, is an $n\times n$
array
\[
  \Phi=(\phi_{ij})_{0\leq i,j\leq n-1}
\]
of unit vectors in $\C^n$ such that every row and every column forms an
orthonormal basis of $\C^n$.
\end{definition}

Two unit vectors represent the same quantum state when they differ only by a
global phase:
\[
  u\sim v
  \quad\Longleftrightarrow\quad
  u=\e^{\ii\alpha}v
  \quad\text{for some }\alpha\in\R.
\]
The cardinality of $\Phi$ is
\[
  \card(\Phi)
  =\#\bigl\{[\phi_{ij}]:0\leq i,j\leq n-1\bigr\}.
\]

\begin{definition}[Complex Hadamard and Butson matrices]
An $n\times n$ matrix $H$ is a complex Hadamard matrix if every entry of $H$
has modulus one and
\[
  H^*H=HH^*=nI_n.
\]
It is \emph{dephased} if its first row and first column consist entirely of
ones.  If every entry is a $q$th root of unity, then $H$ is a Butson matrix,
denoted by $BH(n,q)$.
\end{definition}

Complex Hadamard matrices form a rich class of unitary structures; see
Tadej and {\.{Z}}yczkowski for a survey \cite{TadejZyczkowski2006}.  Their
classification is already nontrivial in order six.  Two order-six matrices
used below are the isolated Tao matrix in $BH(6,3)$ and a Di\c{t}\u{a}--Fourier
type matrix in $BH(6,6)$.

For vectors $x=(x_r)$ and $y=(y_r)$ in $\C^n$, their Schur product is
\[
  x\od y=(x_0y_0,\ldots,x_{n-1}y_{n-1})^{\mathsf T}.
\]

\section{Hadamard-product constructions}\label{sec:hadprod}

\begin{proposition}[Mixed Hadamard products]\label{prop:mixed}
Let $H=(h_0,\ldots,h_{n-1})$ and $K=(k_0,\ldots,k_{n-1})$ be complex
Hadamard matrices of order $n$, where the displayed vectors are columns.
Then the array
\begin{equation}\label{eq:mixedvectors}
  \Phi(H,K)_{ij}=\frac1{\sqrt n}(h_i\od k_j),
  \qquad 0\leq i,j\leq n-1,
\end{equation}
is a $\QLS(n)$.
\end{proposition}

\begin{proof}
Every vector in \eqref{eq:mixedvectors} has norm one.  For fixed $i$,
\[
\begin{aligned}
 \ip{\Phi(H,K)_{ij}}{\Phi(H,K)_{ij'}}
 &=\frac1n\sum_{r=0}^{n-1}
   \overline{(h_i)_r(k_j)_r}(h_i)_r(k_{j'})_r\\
 &=\frac1n\sum_{r=0}^{n-1}\overline{(k_j)_r}(k_{j'})_r
 =\delta_{jj'}.
\end{aligned}
\]
Thus each row is an orthonormal basis.  The same calculation with the roles of
$H$ and $K$ interchanged shows that each column is an orthonormal basis.
\end{proof}

Taking $K=H$ gives the symmetric construction
\[
  \Phi(H)=\Phi(H,H),
  \qquad
  v_{ij}=\frac1{\sqrt n}(h_i\od h_j).
\]

\begin{proposition}\label{prop:symmetric}
For every complex Hadamard matrix $H$ of order $n$, $\Phi(H)$ is a symmetric
$\QLS(n)$.  If $n=6$, then
\[
  \card(\Phi(H))\leq\binom{7}{2}=21.
\]
If $H$ is dephased, two vectors $v_{ij}$ and $v_{k\ell}$ are phase equivalent
if and only if they are equal.
\end{proposition}

\begin{proof}
The quantum Latin property follows from Proposition~\ref{prop:mixed} and
$v_{ij}=v_{ji}$.  Hence there is at most one phase class for each unordered
pair $\{i,j\}$, giving the stated upper bound.  If $H$ is dephased, the first
coordinate of every $v_{ij}$ equals $1/\sqrt n$.  Therefore
$v_{ij}=\lambda v_{k\ell}$ with $|\lambda|=1$ forces $\lambda=1$.
\end{proof}

For compact display of the explicit arrays, let
\[
 \xi_q=\e^{2\pi\ii/q},
 \qquad
 \mathbf r_q(a_1a_2a_3a_4a_5)
 =
 \frac1{\sqrt6}
 \bigl(1,\xi_q^{a_1},\xi_q^{a_2},\xi_q^{a_3},
       \xi_q^{a_4},\xi_q^{a_5}\bigr)^{\mathsf T}.
\]
Thus a five-digit string denotes a fully explicit dephased unit vector.
For $q=12$, the symbols $\mathrm A$ and $\mathrm B$ denote the residues
$10$ and $11$.  If $S=(s_{ij})$ is a $6\times6$ table of such strings, then
$\mathbf r_q[S]$ denotes the array obtained by applying $\mathbf r_q$
entrywise.  This notation fixes the first coordinate to $1/\sqrt6$ and
therefore also fixes the global phase.

\section{A \texorpdfstring{$\QLS(6)$}{QLS(6)} of cardinality 19}\label{sec:c19}

Let
\[
  \zeta=\e^{\pi\ii/4},
\]
and consider the exponent matrix over $\Z_8$
\begin{equation}\label{eq:A19}
A_{19}=
\begin{pmatrix}
0&0&0&0&0&0\\
0&4&2&6&6&2\\
0&2&4&3&7&6\\
0&6&1&4&2&5\\
0&6&5&2&4&1\\
0&2&6&7&3&4
\end{pmatrix}.
\end{equation}
Define
\begin{equation}\label{eq:H19}
  (H_{19})_{rs}=\zeta^{(A_{19})_{rs}},
  \qquad 0\leq r,s\leq5.
\end{equation}
The matrix is dephased and all entries are eighth roots of unity.

\begin{proposition}\label{prop:H19had}
The matrix $H_{19}$ is a complex Hadamard matrix of order $6$.
\end{proposition}

\begin{proof}
For distinct columns $j$ and $k$, their inner product is
\[
  \sum_{r=0}^5 \zeta^{(A_{19})_{rk}-(A_{19})_{rj}}.
\]
The residue multisets of the exponent differences fall into the five patterns
\[
\begin{array}{c|c|c}
\text{pattern}&\text{residues modulo }8&\text{column pairs}\\
\hline
P_1&\{0,2,2,4,6,6\}&01\\
P_2&\{0,1,2,4,5,6\}&02,05,13,14\\
P_3&\{0,2,3,4,6,7\}&03,04,12,15\\
P_4&\{0,1,3,4,5,7\}&23,24,35,45\\
P_5&\{0,0,2,4,4,6\}&25,34.
\end{array}
\]
Each corresponding root sum vanishes:
\[
\begin{aligned}
P_1&: 1+2\zeta^2+\zeta^4+2\zeta^6=0,\\
P_2&: 1+\zeta+\zeta^2+\zeta^4+\zeta^5+\zeta^6=0,\\
P_3&: 1+\zeta^2+\zeta^3+\zeta^4+\zeta^6+\zeta^7=0,\\
P_4&: 1+\zeta+\zeta^3+\zeta^4+\zeta^5+\zeta^7=0,\\
P_5&: 2+\zeta^2+2\zeta^4+\zeta^6=0.
\end{aligned}
\]
Thus $H_{19}^*H_{19}=6I_6$.
\end{proof}

Let $h_0,\ldots,h_5$ be the columns of $H_{19}$ and set
$\Phi_{19}=\Phi(H_{19})$.  In the notation introduced above, the complete
array is the following explicit matrix:
\begin{equation}\label{eq:explicit19}
\Phi_{19}=\mathbf r_8\!\left[
\begin{pmatrix}
00000&42662&24156&63427&67243&26514\\
42662&04444&66730&25201&21025&60376\\
24156&66730&40224&07575&03311&42662\\
63427&25201&07575&46046&42662&01133\\
67243&21025&03311&42662&46406&05757\\
26514&60376&42662&01133&05757&44220
\end{pmatrix}
\right].
\end{equation}
For example, the repeated entry $42662$ represents
\[
 \mathbf r_8(42662)
 =\frac1{\sqrt6}(1,-1,\ii,-\ii,-\ii,\ii)^{\mathsf T}.
\]

\begin{theorem}\label{thm:c19}
The array $\Phi_{19}$ is a $\QLS(6)$ with cardinality $19$.
\end{theorem}

\begin{proof}
Proposition~\ref{prop:symmetric} gives $\Phi_{19}\in\QLS(6)$ and reduces
phase equivalence to equality.  For an unordered pair $0\leq i\leq j\leq5$,
the exponent signature is
\begin{equation}\label{eq:signature19}
  \sigma_{ij}=A_{19,\bullet i}+A_{19,\bullet j}\pmod8.
\end{equation}
The complete explicit array is displayed in \eqref{eq:explicit19}, and the
unordered signature list is recorded in Appendix~\ref{app:c19cert}.  Its only
repetition is
\[
  \sigma_{01}=\sigma_{25}=\sigma_{34},
\]
so $v_{01}=v_{25}=v_{34}$, while all other unordered products are distinct.
Therefore
\[
  \card(\Phi_{19})=21-(3-1)=19.
\]
\end{proof}

For direct inspection, the row-major phase-class matrix is
\begin{equation}\label{eq:class19}
C_{19}=
\begin{pmatrix}
0&1&2&3&4&5\\
1&6&7&8&9&10\\
2&7&11&12&13&1\\
3&8&12&14&1&15\\
4&9&13&1&16&17\\
5&10&1&15&17&18
\end{pmatrix}.
\end{equation}
Its entries are exactly $0,1,\ldots,18$.  The label $1$ occurs in the six
ordered cells corresponding to the three unordered coincidences
$\{0,1\}$, $\{2,5\}$, and $\{3,4\}$.

\section{A \texorpdfstring{$\QLS(6)$}{QLS(6)} of cardinality 21}\label{sec:c21}

Let $\omega=\e^{2\pi\ii/3}$ and consider Tao's isolated $BH(6,3)$ matrix
\cite{TadejZyczkowski2006,Szollosi2012}.  In exponent form it is
\begin{equation}\label{eq:A21}
A_{21}=
\begin{pmatrix}
0&0&0&0&0&0\\
0&0&1&1&2&2\\
0&1&0&2&2&1\\
0&1&2&0&1&2\\
0&2&2&1&0&1\\
0&2&1&2&1&0
\end{pmatrix}
\quad\text{over }\Z_3,
\end{equation}
and
\begin{equation}\label{eq:H21}
  (H_{21})_{rs}=\omega^{(A_{21})_{rs}}.
\end{equation}

\begin{proposition}\label{prop:H21had}
The matrix $H_{21}$ is a dephased complex Hadamard matrix.
\end{proposition}

\begin{proof}
Every entry has modulus one.  For every pair of distinct columns, the six
coordinatewise exponent differences contain each residue $0$, $1$, and $2$
exactly twice.  Hence every off-diagonal column inner product equals
\[
  2(1+\omega+\omega^2)=0.
\]
Thus $H_{21}^*H_{21}=6I_6$.
\end{proof}

Let $\Phi_{21}=\Phi(H_{21})$.  Its complete explicit form is
\begin{equation}\label{eq:explicit21}
\Phi_{21}=\mathbf r_3\!\left[
\begin{pmatrix}
00000&01122&10221&12012&22101&21210\\
01122&02211&11010&10101&20220&22002\\
10221&11010&20112&22200&02022&01101\\
12012&10101&22200&21021&01110&00222\\
22101&20220&02022&01110&11202&10011\\
21210&22002&01101&00222&10011&12120
\end{pmatrix}
\right].
\end{equation}

\begin{theorem}\label{thm:c21}
The array $\Phi_{21}$ is a $\QLS(6)$ with cardinality $21$.
\end{theorem}

\begin{proof}
By Proposition~\ref{prop:symmetric}, $\Phi_{21}$ is a symmetric $\QLS(6)$
and has at most $21$ phase classes.  Since $H_{21}$ is dephased, phase
equivalence is equality.  For $0\leq i\leq j\leq5$, define the exponent
signature
\[
  \tau_{ij}=A_{21,\bullet i}+A_{21,\bullet j}\pmod3.
\]
Equation~\eqref{eq:explicit21} displays all thirty-six cells explicitly,
while Appendix~\ref{app:c21cert} lists the twenty-one unordered signatures.
They are pairwise
distinct in $\Z_3^6$, so all unordered Schur products are distinct.  Hence
\[
  \card(\Phi_{21})=21.
\]
\end{proof}

A row-major phase-class certificate is
\begin{equation}\label{eq:class21}
C_{21}=
\begin{pmatrix}
0&1&2&3&4&5\\
1&6&7&8&9&10\\
2&7&11&12&13&14\\
3&8&12&15&16&17\\
4&9&13&16&18&19\\
5&10&14&17&19&20
\end{pmatrix}.
\end{equation}
The labels in \eqref{eq:class21} are exactly $0,1,\ldots,20$.

\section{A direct-sum family of cardinalities 23 and 25}\label{sec:direct}

Write
\[
 \C^6=\mathcal U\oplus\mathcal V,
 \qquad
 \mathcal U=\operatorname{span}\{e_0,e_1,e_2,e_3\},
 \qquad
 \mathcal V=\operatorname{span}\{e_4,e_5\},
\]
and let $\rho=1/\sqrt2$.  Choose real numbers $a,b>0$ satisfying
\[
 a^2+b^2=1.
\]
Define the following unit vectors in $\mathcal U\cong\R^4$:
\begin{align*}
 A&=(1,0,0,0),&
 B&=(0,1,0,0),&
 C&=(0,0,1,0),&
 D&=(0,0,0,1),\\
 E&=(0,a,b\rho,b\rho),&
 F&=(a,0,b\rho,-b\rho),&
 G&=(b\rho,-b\rho,0,a),&
 H&=(-b\rho,-b\rho,a,0),\\
 I&=(0,-b,a\rho,a\rho),&
 J&=(-b,0,a\rho,-a\rho),&
 K&=(-a\rho,a\rho,0,b),&
 L&=(a\rho,a\rho,b,0),\\
 M&=(-a\rho,a\rho,b,0),&
 N&=(b\rho,-b\rho,a,0),&
 O&=(\rho,\rho,0,0),&
 P&=(a\rho,a\rho,0,b),\\
 Q&=(-b\rho,-b\rho,0,a),&
 R&=(-\rho,\rho,0,0),&
 S&=(0,0,\rho,\rho),&
 T&=(b,0,a\rho,-a\rho),\\
 U&=(-a,0,b\rho,-b\rho).&&&&
\end{align*}
In $\mathcal V$, define
\[
 x=e_4,
 \qquad y=e_5,
 \qquad z=\rho(e_4+e_5),
 \qquad w=\rho(e_4-e_5).
\]

\begin{proposition}\label{prop:directbases}
For every $a,b>0$ with $a^2+b^2=1$, each of the following sets is an
orthonormal basis of $\mathcal U$:
\[
\begin{array}{llll}
\{A,B,C,D\}, & \{E,F,G,H\}, & \{I,J,K,L\}, & \{M,D,N,O\},\\
\{P,C,Q,R\}, & \{T,B,U,S\}, & \{I,M,P,T\}, & \{A,D,C,B\},\\
\{E,N,Q,U\}, & \{B,F,J,S\}, & \{C,G,K,O\}, & \{D,H,L,R\}.
\end{array}
\]
\end{proposition}

\begin{proof}
The assertion follows by exact dot products.  For example, the matrix with
columns $E,F,G,H$ is
\[
 \begin{pmatrix}
 0&a&b\rho&-b\rho\\
 a&0&-b\rho&-b\rho\\
 b\rho&b\rho&0&a\\
 b\rho&-b\rho&a&0
 \end{pmatrix},
\]
whose Gram matrix is $(a^2+b^2)I_4=I_4$.  Direct multiplication gives $I_4$
for each of the other eleven displayed sets.
\end{proof}

Define
\begin{equation}\label{eq:Phiab}
 \Phi(a,b)=
 \begin{pmatrix}
 x&A&y&B&C&D\\
 y&x&E&F&G&H\\
 I&y&x&J&K&L\\
 M&D&N&z&O&w\\
 P&C&Q&w&z&R\\
 T&B&U&S&w&z
 \end{pmatrix}.
\end{equation}

\begin{proposition}\label{prop:PhiabQLS}
For every $a,b>0$ with $a^2+b^2=1$, the array $\Phi(a,b)$ is a $\QLS(6)$.
\end{proposition}

\begin{proof}
Every row and column contains four vectors in $\mathcal U$ and two vectors in
$\mathcal V$.  The twelve four-dimensional parts, in row order and then
column order, are exactly the bases in Proposition~\ref{prop:directbases}.
The two-dimensional part is either $\{x,y\}$ or $\{z,w\}$.  Since
$\mathcal U\perp\mathcal V$, every row and column is an orthonormal basis of
$\C^6$.
\end{proof}

\subsection{The symmetric point: cardinality 23}

Set
\[
 a=b=\rho=\frac1{\sqrt2},
 \qquad
 \Phi_{23}=\Phi(\rho,\rho).
\]
Substitution into the vectors defining \eqref{eq:Phiab} gives the following
explicit representatives in $\mathcal U$:
\begin{align}
E&=(0,\rho,\tfrac12,\tfrac12),&
F&=(\rho,0,\tfrac12,-\tfrac12),&
G&=(\tfrac12,-\tfrac12,0,\rho),&
H&=(-\tfrac12,-\tfrac12,\rho,0),\notag\\
I&=(0,-\rho,\tfrac12,\tfrac12),&
J&=(-\rho,0,\tfrac12,-\tfrac12),&
K&=(-\tfrac12,\tfrac12,0,\rho),&
L&=(\tfrac12,\tfrac12,\rho,0),\notag\\
M&=(-\tfrac12,\tfrac12,\rho,0),&
N&=(\tfrac12,-\tfrac12,\rho,0),&
O&=(\rho,\rho,0,0),&
P&=(\tfrac12,\tfrac12,0,\rho),\notag\\
Q&=(-\tfrac12,-\tfrac12,0,\rho),&
R&=(-\rho,\rho,0,0),&
S&=(0,0,\rho,\rho),&
T&=(\rho,0,\tfrac12,-\tfrac12),\notag\\
U&=(-\rho,0,\tfrac12,-\tfrac12).
\label{eq:explicit23vectors}
\end{align}
Together with $A,B,C,D$ and $x,y,z,w$ defined above, these vectors inserted
entrywise into \eqref{eq:Phiab} are the complete explicit matrix
$\Phi_{23}$.

\begin{theorem}\label{thm:c23}
The array $\Phi_{23}$ is a $\QLS(6)$ with cardinality $23$.
\end{theorem}

\begin{proof}
The quantum Latin property follows from Proposition~\ref{prop:PhiabQLS}.
At the symmetric point,
\[
 F=T,
 \qquad
 J=U.
\]
A comparison of the remaining labeled vectors gives
\[
 \bigl\{|\ip{X}{Y}|:X\neq Y,\ X,Y\in\{A,\ldots,U\}\bigr\}
 =\left\{0,\frac12,\frac1{\sqrt2},1\right\},
\]
and the value $1$ occurs only for the pairs $\{F,T\}$ and $\{J,U\}$.
Consequently the twenty-one labels $A,\ldots,U$ determine nineteen rays in
$\mathcal U$.  The vectors $x,y,z,w$ determine four distinct rays in
$\mathcal V$, and no nonzero ray can lie in both orthogonal summands.  Thus
\[
 \card(\Phi_{23})=19+4=23.
\]
\end{proof}

A phase-class certificate is
\begin{equation}\label{eq:class23}
C_{23}=
\begin{pmatrix}
0&1&2&3&4&5\\
2&0&6&7&8&9\\
10&2&0&11&12&13\\
14&5&15&16&17&18\\
19&4&20&18&16&21\\
7&3&11&22&18&16
\end{pmatrix}.
\end{equation}
The labels in \eqref{eq:class23} are exactly $0,1,\ldots,22$.

\subsection{Ray splitting: cardinality 25}

Set
\[
 a=\frac45,
 \qquad
 b=\frac35,
 \qquad
 \Phi_{25}=\Phi\left(\frac45,\frac35\right).
\]
For readability, put
\[
 \alpha=\frac{3}{5\sqrt2},
 \qquad
 \beta=\frac{4}{5\sqrt2}.
\]
The specialized vectors in $\mathcal U$ are explicitly
\begin{align}
E&=(0,\tfrac45,\alpha,\alpha),&
F&=(\tfrac45,0,\alpha,-\alpha),&
G&=(\alpha,-\alpha,0,\tfrac45),&
H&=(-\alpha,-\alpha,\tfrac45,0),\notag\\
I&=(0,-\tfrac35,\beta,\beta),&
J&=(-\tfrac35,0,\beta,-\beta),&
K&=(-\beta,\beta,0,\tfrac35),&
L&=(\beta,\beta,\tfrac35,0),\notag\\
M&=(-\beta,\beta,\tfrac35,0),&
N&=(\alpha,-\alpha,\tfrac45,0),&
O&=(\rho,\rho,0,0),&
P&=(\beta,\beta,0,\tfrac35),\notag\\
Q&=(-\alpha,-\alpha,0,\tfrac45),&
R&=(-\rho,\rho,0,0),&
S&=(0,0,\rho,\rho),&
T&=(\tfrac35,0,\beta,-\beta),\notag\\
U&=(-\tfrac45,0,\alpha,-\alpha).
\label{eq:explicit25vectors}
\end{align}
Together with $A,B,C,D$ and $x,y,z,w$, substitution of these vectors into
\eqref{eq:Phiab} gives the complete explicit matrix $\Phi_{25}$.

\begin{theorem}\label{thm:c25}
The array $\Phi_{25}$ is a $\QLS(6)$ with cardinality $25$.
\end{theorem}

\begin{proof}
Again the quantum Latin property follows from
Proposition~\ref{prop:PhiabQLS}.  For the chosen rational parameters, direct
calculation of the off-diagonal inner products among $A,\ldots,U$ gives
\begin{align*}
 \bigl\{|\ip{X}{Y}|:X\ne Y,\ X,Y\in\{A,\ldots,U\}\bigr\}
 =\biggl\{&0,\frac{7\sqrt2}{50},\frac7{25},\frac9{25},
 \frac{3\sqrt2}{10},\frac{12}{25},\frac{2\sqrt2}{5},\\
 &\frac35,\frac{16}{25},\frac{12\sqrt2}{25},
 \frac1{\sqrt2},\frac45,\frac{24}{25}\biggr\}.
\end{align*}
In particular,
\[
 \max_{X\ne Y}|\ip{X}{Y}|=\frac{24}{25}<1,
\]
so $A,\ldots,U$ determine twenty-one distinct rays in $\mathcal U$.  Together
with the four rays $x,y,z,w$ in $\mathcal V$, this gives
\[
 \card(\Phi_{25})=21+4=25.
\]
\end{proof}

A phase-class certificate is
\begin{equation}\label{eq:class25}
C_{25}=
\begin{pmatrix}
0&1&2&3&4&5\\
2&0&6&7&8&9\\
10&2&0&11&12&13\\
14&5&15&16&17&18\\
19&4&20&18&16&21\\
22&3&23&24&18&16
\end{pmatrix}.
\end{equation}
The labels in \eqref{eq:class25} are exactly $0,1,\ldots,24$.

\begin{remark}[Controlled deformation]
The incidence pattern and all twelve orthonormal bases in \eqref{eq:Phiab}
remain fixed throughout the circle $a^2+b^2=1$.  Passing from
$a=b=1/\sqrt2$ to $(a,b)=(4/5,3/5)$ splits precisely the coincidences
$F=T$ and $J=U$, raising the cardinality from $23$ to $25$.
\end{remark}

\section{A \texorpdfstring{$\QLS(6)$}{QLS(6)} of cardinality 27}\label{sec:c27}

Let
\[
 \omega=\e^{2\pi\ii/3},
 \qquad
 F_3=\begin{pmatrix}
 1&1&1\\
 1&\omega&\omega^2\\
 1&\omega^2&\omega
 \end{pmatrix},
 \qquad
 D=\operatorname{diag}(1,1,\omega),
\]
and define
\begin{equation}\label{eq:H27}
 H_{27}=
 \begin{pmatrix}
  F_3&D F_3\\
  F_3&-D F_3
 \end{pmatrix}
 =
 \begin{pmatrix}
 1&1&1&1&1&1\\
 1&\omega&\omega^2&1&\omega&\omega^2\\
 1&\omega^2&\omega&\omega&1&\omega^2\\
 1&1&1&-1&-1&-1\\
 1&\omega&\omega^2&-1&-\omega&-\omega^2\\
 1&\omega^2&\omega&-\omega&-1&-\omega^2
 \end{pmatrix}.
\end{equation}
Since $F_3F_3^*=3I_3$ and $D$ is unitary, block multiplication gives
$H_{27}H_{27}^*=6I_6$.  Hence $H_{27}$ is a $BH(6,6)$ matrix.

Let $P$ exchange the first two rows and fix the remaining rows, and set
\[
 K_{27}=PH_{27}.
\]
Let $h_0,\ldots,h_5$ and $k_0,\ldots,k_5$ be the columns of $H_{27}$ and
$K_{27}$, respectively, and define
\begin{equation}\label{eq:Phi27}
 \phi_{ij}=\frac1{\sqrt6}(h_i\od k_j),
 \qquad
 \Phi_{27}=(\phi_{ij})_{0\leq i,j\leq5}.
\end{equation}
By Proposition~\ref{prop:mixed}, $\Phi_{27}$ is a $\QLS(6)$.

Set
\[
 d=(1,1,\omega,-1,-1,-\omega)^{\mathsf T}.
\]
Inspection of \eqref{eq:H27} gives
\[
 h_{i+3}=d\od h_i,
 \qquad 0\leq i\leq2.
\]
Because the first two entries of $d$ are equal, $Pd=d$, and therefore
\[
 k_{j+3}=d\od k_j,
 \qquad 0\leq j\leq2.
\]
It follows that
\begin{equation}\label{eq:nineequalities}
 \phi_{i+3,j}=\phi_{i,j+3},
 \qquad 0\leq i,j\leq2.
\end{equation}
Thus the lower-left $3\times3$ block repeats the upper-right $3\times3$ block.

To exclude additional phase coincidences, let $\xi=\e^{\pi\ii/3}$ and write
$(H_{27})_{rs}=\xi^{A_{rs}}$ with exponent matrix
\begin{equation}\label{eq:A27}
 A_{27}=
 \begin{pmatrix}
 0&0&0&0&0&0\\
 0&2&4&0&2&4\\
 0&4&2&2&0&4\\
 0&0&0&3&3&3\\
 0&2&4&3&5&1\\
 0&4&2&5&3&1
 \end{pmatrix}
 \quad\text{over }\Z_6.
\end{equation}
Let $p=(0\ 1)$ denote the row transposition.  Before normalization, the
exponent in coordinate $r$ of $\phi_{ij}$ is
\[
 s_{ij,r}=A_{27,ri}+A_{27,p(r),j}\pmod6.
\]
Subtract the first exponent to remove global phase and define
\[
 \sigma_{ij}=
 (s_{ij,1}-s_{ij,0},\ldots,s_{ij,5}-s_{ij,0})\in\Z_6^5.
\]
Writing a five-tuple as a five-digit string, the complete signature table is
\begin{equation}\label{eq:signatures27}
\renewcommand{\arraystretch}{1.12}
 \begin{pmatrix}
 00000&42402&24204&02335&44131&20533\\
 24024&00420&42222&20353&02155&44551\\
 42042&24444&00240&44311&20113&02515\\
 02335&44131&20533&04004&40400&22202\\
 20353&02155&44551&22022&04424&40220\\
 44311&20113&02515&40040&22442&04244
 \end{pmatrix}.
\end{equation}
Consequently, \eqref{eq:signatures27} is not merely a collision certificate:
it gives the complete explicit QLS
\begin{equation}\label{eq:explicit27}
 \Phi_{27}=\mathbf r_6\!\left[\Sigma_{27}\right],
\end{equation}
where $\Sigma_{27}$ is the $6\times6$ string matrix displayed in
\eqref{eq:signatures27}.  The table has exactly nine repeated signatures,
namely the pairs in
\eqref{eq:nineequalities}; the other eighteen signatures occur once.
Equivalently, a phase-class labeling is
\begin{equation}\label{eq:class27}
 C_{27}=
 \begin{pmatrix}
 0&1&2&3&4&5\\
 6&7&8&9&10&11\\
 12&13&14&15&16&17\\
 3&4&5&18&19&20\\
 9&10&11&21&22&23\\
 15&16&17&24&25&26
 \end{pmatrix}.
\end{equation}

\begin{theorem}\label{thm:c27}
The array $\Phi_{27}$ is a $\QLS(6)$ with cardinality $27$.
\end{theorem}

\begin{proof}
The quantum Latin property follows from Proposition~\ref{prop:mixed}.  The
signature certificate consists of nine two-element phase classes and eighteen
singleton classes.  Therefore
\[
 \card(\Phi_{27})=9+18=27.
\]
\end{proof}

\begin{remark}[Row-permutation classification]
For the fixed matrix $H_{27}$, enumeration of all $720$ row permutations
$P\in S_6$ gives
\[
 \begin{array}{c|rrrr}
 \card\Phi(H_{27},PH_{27})&9&12&27&36\\ \hline
 \#\{P\in S_6\}&12&36&36&636.
 \end{array}
\]
Thus cardinality $27$ occurs for thirty-six row permutations of this one
$BH(6,6)$ matrix.
\end{remark}

\section{A \texorpdfstring{$\QLS(6)$}{QLS(6)} of cardinality 32}\label{sec:c32}

Let $\eta=\e^{\pi\ii/3}$.  Consider the exponent matrices over $\Z_6$
\begin{equation}\label{eq:A32B32}
 A_{32}=
 \begin{pmatrix}
 0&0&0&0&0&0\\
 0&0&2&2&4&4\\
 0&2&0&4&2&4\\
 0&2&4&0&4&2\\
 0&4&2&4&0&2\\
 0&4&4&2&2&0
 \end{pmatrix},
 \qquad
 B_{32}=
 \begin{pmatrix}
 0&0&0&0&0&0\\
 0&0&2&2&4&4\\
 0&0&4&4&2&2\\
 0&3&0&3&1&4\\
 0&3&2&5&5&2\\
 0&3&4&1&3&0
 \end{pmatrix}.
\end{equation}
Define
\[
 (H_{32})_{rs}=\eta^{(A_{32})_{rs}},
 \qquad
 (K_{32})_{rs}=\eta^{(B_{32})_{rs}}.
\]
The matrix $H_{32}$ is Tao's $BH(6,3)$ matrix written with even exponents
modulo $6$.

\begin{proposition}\label{prop:H32K32had}
Both $H_{32}$ and $K_{32}$ are dephased complex Hadamard matrices.
\end{proposition}

\begin{proof}
For $H_{32}$, every difference of two distinct columns has residue multiset
\[
 \{0,0,2,2,4,4\},
\]
whose root sum is $2(1+\eta^2+\eta^4)=0$.  For $K_{32}$, the difference
multisets fall into the three patterns
\[
\begin{array}{c|c|c}
\text{pattern}&\text{residues modulo }6&\text{column pairs}\\
\hline
Q_1&\{0,0,0,3,3,3\}&01,23,45\\
Q_2&\{0,0,2,2,4,4\}&02,05,13,14,25,34\\
Q_3&\{0,1,2,3,4,5\}&03,04,12,15,24,35.
\end{array}
\]
The corresponding sums are
$3(1+\eta^3)$, $2(1+\eta^2+\eta^4)$, and
$\sum_{r=0}^{5}\eta^r$, respectively, and all vanish.  Hence both Gram
matrices equal $6I_6$.
\end{proof}

Let $h_0,\ldots,h_5$ and $k_0,\ldots,k_5$ be the columns of $H_{32}$ and
$K_{32}$, and define
\[
 \Phi_{32,ij}=\frac1{\sqrt6}(h_i\od k_j),
 \qquad
 \Phi_{32}=(\Phi_{32,ij})_{0\leq i,j\leq5}.
\]
By Proposition~\ref{prop:mixed}, $\Phi_{32}$ is a $\QLS(6)$.

For each cell, set
\[
 s_{ij,r}=(A_{32})_{ri}+(B_{32})_{rj}\pmod6,
 \qquad
 \sigma_{ij}=(s_{ij,1}-s_{ij,0},\ldots,s_{ij,5}-s_{ij,0})\in\Z_6^5.
\]
The complete signature table is
\begin{equation}\label{eq:signatures32}
\renewcommand{\arraystretch}{1.12}
 \begin{pmatrix}
 00000&00333&24024&24351&42153&42420\\
 02244&02511&20202&20535&44331&44004\\
 20424&20151&44442&44115&02511&02244\\
 24042&24315&42000&42333&00135&00402\\
 42402&42135&00420&00153&24555&24222\\
 44220&44553&02244&02511&20313&20040
 \end{pmatrix}.
\end{equation}
Applying $\mathbf r_6$ entrywise to this table gives the complete explicit
array:
\begin{equation}\label{eq:explicit32}
 \Phi_{32}=\mathbf r_6\!\left[\Sigma_{32}\right],
\end{equation}
where $\Sigma_{32}$ denotes the string matrix in
\eqref{eq:signatures32}.  Here each five-digit string denotes an element of
$\Z_6^5$.  Exactly two
signatures occur more than once:
\begin{align*}
 02244&:\quad (1,0),(2,5),(5,2),\\
 02511&:\quad (1,1),(2,4),(5,3).
\end{align*}
All remaining thirty signatures occur once.

\begin{theorem}\label{thm:c32}
The array $\Phi_{32}$ is a $\QLS(6)$ with cardinality $32$.
\end{theorem}

\begin{proof}
The quantum Latin property follows from Proposition~\ref{prop:mixed}.  Since
both exponent matrices are dephased, the first coordinate of every vector in
$\Phi_{32}$ is $1/\sqrt6$.  Thus two cells represent the same ray if and only
if their dephased signatures agree.  Table~\eqref{eq:signatures32} has two
three-element phase classes and thirty singleton classes, so
\[
 \card(\Phi_{32})=2+30=32.
\]
\end{proof}

The corresponding phase-class matrix is
\begin{equation}\label{eq:class32}
C_{32}=
\begin{pmatrix}
0&1&2&3&4&5\\
6&7&8&9&10&11\\
12&13&14&15&7&6\\
16&17&18&19&20&21\\
22&23&24&25&26&27\\
28&29&6&7&30&31
\end{pmatrix}.
\end{equation}
The labels $6$ and $7$ each occur three times, and every other label occurs
once.

\section{A \texorpdfstring{$\QLS(6)$}{QLS(6)} of cardinality 35}\label{sec:c35}

Let $\theta=\e^{\pi\ii/6}$ and consider the exponent matrix over $\Z_{12}$
\begin{equation}\label{eq:A35}
 A_{35}=
 \begin{pmatrix}
 0&0&0&0&0&0\\
 0&0&3&6&6&9\\
 0&3&9&1&7&6\\
 0&6&4&10&1&7\\
 0&6&10&7&4&1\\
 0&9&6&4&10&3
 \end{pmatrix}.
\end{equation}
Define $(H_{35})_{rs}=\theta^{(A_{35})_{rs}}$.  Let $P$ exchange rows $4$
and $5$ and fix rows $0,1,2,3$, and set
\begin{equation}\label{eq:K35}
 K_{35}=P H_{35}^{\mathsf T}.
\end{equation}
Equivalently, the exponent matrix of $K_{35}$ is
\begin{equation}\label{eq:B35}
 B_{35}=
 \begin{pmatrix}
 0&0&0&0&0&0\\
 0&0&3&6&6&9\\
 0&3&9&4&10&6\\
 0&6&1&10&7&4\\
 0&9&6&7&1&3\\
 0&6&7&1&4&10
 \end{pmatrix}.
\end{equation}

\begin{proposition}\label{prop:H35K35had}
Both $H_{35}$ and $K_{35}$ are dephased complex Hadamard matrices.
\end{proposition}

\begin{proof}
For $H_{35}$, the difference multisets of distinct columns are
\[
\begin{array}{c|c|c}
\text{pattern}&\text{residues modulo }12&\text{column pairs}\\
\hline
R_1&\{0,0,3,6,6,9\}&01,34\\
R_2&\{0,3,4,6,9,10\}&02,12,23,24\\
R_3&\{0,1,4,6,7,10\}&03,04,13,14\\
R_4&\{0,1,3,6,7,9\}&05,15\\
R_5&\{0,3,3,6,9,9\}&25\\
R_6&\{0,3,5,6,9,11\}&35,45.
\end{array}
\]
Every displayed multiset is a union of three antipodal pairs
$\{t,t+6\}$ modulo $12$.  Since $\theta^{t+6}=-\theta^t$, every off-diagonal
column inner product vanishes.  Thus $H_{35}$ is Hadamard.  Transposition and
row permutation preserve the Hadamard property, so $K_{35}$ is Hadamard as
well.  Both matrices are visibly dephased.
\end{proof}

Let $h_0,\ldots,h_5$ and $k_0,\ldots,k_5$ be the columns of $H_{35}$ and
$K_{35}$, and set
\[
 \Phi_{35,ij}=\frac1{\sqrt6}(h_i\od k_j),
 \qquad
 \Phi_{35}=(\Phi_{35,ij})_{0\leq i,j\leq5}.
\]
Again, Proposition~\ref{prop:mixed} gives $\Phi_{35}\in\QLS(6)$.

Define the dephased signatures $\rho_{ij}\in\Z_{12}^5$ from $A_{35}$ and
$B_{35}$ exactly as above.  In the following table, the symbols
$\mathrm A$ and $\mathrm B$ denote the residues $10$ and $11$:
\begin{equation}\label{eq:signatures35}
\renewcommand{\arraystretch}{1.12}
 \begin{pmatrix}
 00000&03696&39167&64\mathrm A71&6\mathrm A714&9643\mathrm A\\
 03669&06033&30704&6741\mathrm A&61171&99\mathrm A97\\
 394\mathrm A6&30\mathrm A70&66541&91257&97\mathrm{BB}\mathrm A&03814\\
 61\mathrm A74&6444\mathrm A&9\mathrm{AB}1\mathrm B&05825&0\mathrm B588&372\mathrm A2\\
 6714\mathrm A&6\mathrm A714&942\mathrm A5&0\mathrm{BBBB}&05852&31578\\
 96713&991\mathrm A9&0387\mathrm A&3\mathrm A584&34227&60\mathrm B41
 \end{pmatrix}.
\end{equation}
Thus the full explicit vector array is
\begin{equation}\label{eq:explicit35}
 \Phi_{35}=\mathbf r_{12}\!\left[\Sigma_{35}\right],
\end{equation}
where $\Sigma_{35}$ is the string matrix in \eqref{eq:signatures35}.
Only the signature $6\mathrm A714$ is repeated, at the cells $(0,4)$ and
$(4,1)$; the other thirty-four signatures occur once.

\begin{theorem}\label{thm:c35}
The array $\Phi_{35}$ is a $\QLS(6)$ with cardinality $35$.
\end{theorem}

\begin{proof}
The quantum Latin property follows from Proposition~\ref{prop:mixed}.  Since
$H_{35}$ and $K_{35}$ are dephased, equality of rays is equivalent to equality
of the signatures in \eqref{eq:signatures35}.  There is one two-element phase
class and thirty-four singleton classes.  Therefore
\[
 \card(\Phi_{35})=1+34=35.
\]
\end{proof}

Equivalently, the phase-class matrix is
\begin{equation}\label{eq:class35}
C_{35}=
\begin{pmatrix}
0&1&2&3&4&5\\
6&7&8&9&10&11\\
12&13&14&15&16&17\\
18&19&20&21&22&23\\
24&4&25&26&27&28\\
29&30&31&32&33&34
\end{pmatrix}.
\end{equation}
Only the label $4$ is repeated, at the cells $(0,4)$ and $(4,1)$.

\section{Consequences for the order-six spectrum}\label{sec:spectrum}

Combining the known order-six cardinalities summarized in
\cite[Table~6]{ZhangLvCao2026} with Theorems~\ref{thm:c23},
\ref{thm:c25}, \ref{thm:c27}, \ref{thm:c32}, and \ref{thm:c35} supplies the
additional cases $23$, $25$, $27$, $32$, and $35$.

\begin{corollary}\label{cor:order6spectrum}
For every
\[
 c\in[6,36]\setminus\{7,29\},
\]
there exists a quantum Latin square of order $6$ with cardinality $c$.
Moreover, no $\QLS(6)$ has cardinality $7$.
\end{corollary}

\begin{proof}
The previously attainable values are included in the summary cited above,
and the five additional values are supplied by the present paper.  The
general nonexistence theorem for cardinality $n+1$ excludes $7$ when $n=6$
\cite{ZangEtAl2025,ZhangJi2026}.
\end{proof}

Thus the only order-six cardinality whose existence remains undetermined is
\[
 \boxed{29}.
\]

\section{Discussion}

The seven cardinalities are realized by three mechanisms.  For cardinality
$19$, arithmetic in $\Z_8$ creates one prescribed triple coincidence among
the twenty-one unordered Schur products.  For cardinality $21$, Tao's
$BH(6,3)$ matrix behaves as an additive Sidon-type configuration: all
twenty-one column-sum signatures are distinct.  This shows directly that the
symmetric Schur-product upper bound is attained.

The direct-sum family shows how cardinality can be varied while preserving the
entire row-column incidence pattern.  At the symmetric point, two pairs of
four-dimensional rays coincide, giving $19+4=23$ rays.  The rational
$3$--$4$--$5$ parameter point separates both pairs and gives $21+4=25$ rays.
Thus cardinality is controlled by ray splitting rather than by changing the
linked orthonormal bases.

The mixed Hadamard construction removes the commutativity constraint
$v_{ij}=v_{ji}$ and realizes three qualitatively different collision patterns.
For cardinality $27$, a row permutation forces nine two-element classes.  For
cardinality $32$, two signatures each occur three times and all other cells are
singletons.  For cardinality $35$, a $BH(6,12)$ matrix and a row-permuted
transpose leave exactly one repeated pair.  The last construction is one ray
short of maximal cardinality and illustrates how sparse the collision pattern
can be in a mixed Butson product.

All ray counts are exact finite certificates in residue rings; numerical
linear algebra is not used in the proofs.  The sole remaining order-six case,
cardinality $29$, therefore requires a collision pattern not supplied by the
constructions presented here.

\section{Conclusion}

We have constructed quantum Latin squares of order six with cardinalities
$19$, $21$, $23$, $25$, $27$, $32$, and $35$.  Each construction has an exact
finite certificate: modular exponent signatures for $19$, $21$, $27$, $32$,
and $35$, and inner-product calculations in $\mathbb Q(\sqrt2)$ for $23$ and
$25$.  In combination with the spectrum status reported by Zhang, Lv, and Cao,
the new cardinalities settle five previously open order-six cases and leave
$29$ as the only unresolved cardinality.

\section*{Acknowledgments}

The author gratefully acknowledges support from the Natural Science Research
Project of Jiangsu Higher Education Institutions of China under Grant No.
24KJB520033.  The author also acknowledges the High Performance Computing
Center of the School of Mathematics and Statistics, Nantong University, and
the National Supercomputing Center in Kunshan for computational support.

\appendix

\section{Signature certificate for cardinality 19}\label{app:c19cert}

For each unordered pair $0\leq i\leq j\leq5$, Table~\ref{tab:sig19} lists the
signature \eqref{eq:signature19}.  The single repeated signature occurs for
$(0,1)$, $(2,5)$, and $(3,4)$.

\begin{longtable}{c@{\qquad}c}
\caption{Exponent signatures for the $21$ unordered Schur products of
$H_{19}$.}\label{tab:sig19}\\
\toprule
$(i,j)$ & $\sigma_{ij}$\\
\midrule
\endfirsthead
\toprule
$(i,j)$ & $\sigma_{ij}$\\
\midrule
\endhead
\bottomrule
\endfoot
$(0,0)$ & $(0,0,0,0,0,0)$\\
$(0,1)$ & $(0,4,2,6,6,2)$\\
$(0,2)$ & $(0,2,4,1,5,6)$\\
$(0,3)$ & $(0,6,3,4,2,7)$\\
$(0,4)$ & $(0,6,7,2,4,3)$\\
$(0,5)$ & $(0,2,6,5,1,4)$\\
$(1,1)$ & $(0,0,4,4,4,4)$\\
$(1,2)$ & $(0,6,6,7,3,0)$\\
$(1,3)$ & $(0,2,5,2,0,1)$\\
$(1,4)$ & $(0,2,1,0,2,5)$\\
$(1,5)$ & $(0,6,0,3,7,6)$\\
$(2,2)$ & $(0,4,0,2,2,4)$\\
$(2,3)$ & $(0,0,7,5,7,5)$\\
$(2,4)$ & $(0,0,3,3,1,1)$\\
$(2,5)$ & $(0,4,2,6,6,2)$\\
$(3,3)$ & $(0,4,6,0,4,6)$\\
$(3,4)$ & $(0,4,2,6,6,2)$\\
$(3,5)$ & $(0,0,1,1,3,3)$\\
$(4,4)$ & $(0,4,6,4,0,6)$\\
$(4,5)$ & $(0,0,5,7,5,7)$\\
$(5,5)$ & $(0,4,4,2,2,0)$\\
\end{longtable}

\section{Signature certificate for cardinality 21}\label{app:c21cert}

For each unordered pair $0\leq i\leq j\leq5$, Table~\ref{tab:sig21} lists
\[
 \tau_{ij}=A_{21,\bullet i}+A_{21,\bullet j}\pmod3.
\]
All twenty-one rows are distinct, providing a finite certificate for
Theorem~\ref{thm:c21}.

\begin{longtable}{c@{\qquad}c}
\caption{Exponent signatures for the $21$ unordered Schur products of
Tao's matrix $H_{21}$.}\label{tab:sig21}\\
\toprule
$(i,j)$ & $\tau_{ij}$\\
\midrule
\endfirsthead
\toprule
$(i,j)$ & $\tau_{ij}$\\
\midrule
\endhead
\bottomrule
\endfoot
$(0,0)$ & $(0,0,0,0,0,0)$\\
$(0,1)$ & $(0,0,1,1,2,2)$\\
$(0,2)$ & $(0,1,0,2,2,1)$\\
$(0,3)$ & $(0,1,2,0,1,2)$\\
$(0,4)$ & $(0,2,2,1,0,1)$\\
$(0,5)$ & $(0,2,1,2,1,0)$\\
$(1,1)$ & $(0,0,2,2,1,1)$\\
$(1,2)$ & $(0,1,1,0,1,0)$\\
$(1,3)$ & $(0,1,0,1,0,1)$\\
$(1,4)$ & $(0,2,0,2,2,0)$\\
$(1,5)$ & $(0,2,2,0,0,2)$\\
$(2,2)$ & $(0,2,0,1,1,2)$\\
$(2,3)$ & $(0,2,2,2,0,0)$\\
$(2,4)$ & $(0,0,2,0,2,2)$\\
$(2,5)$ & $(0,0,1,1,0,1)$\\
$(3,3)$ & $(0,2,1,0,2,1)$\\
$(3,4)$ & $(0,0,1,1,1,0)$\\
$(3,5)$ & $(0,0,0,2,2,2)$\\
$(4,4)$ & $(0,1,1,2,0,2)$\\
$(4,5)$ & $(0,1,0,0,1,1)$\\
$(5,5)$ & $(0,1,2,1,2,0)$\\
\end{longtable}


\begin{thebibliography}{99}
\small

\bibitem{Euler1782}
L. Euler,
\newblock \emph{Recherches sur une nouvelle esp\`ece de quarr\'es magiques},
\newblock Verhandelingen uitgegeven door het Zeeuwsch Genootschap der
Wetenschappen te Vlissingen \textbf{9}, 85--239, 1782.

\bibitem{DenesKeedwell1974}
J. D\'enes and A. D. Keedwell,
\newblock \emph{Latin Squares and Their Applications},
\newblock Akad\'emiai Kiad\'o, Budapest, 1974.

\bibitem{ColbournDinitz2007}
C. J. Colbourn and J. H. Dinitz, editors,
\newblock \emph{Handbook of Combinatorial Designs}, second edition,
\newblock Chapman \& Hall/CRC, Boca Raton, 2007.

\bibitem{Werner2001}
R. F. Werner,
\newblock All teleportation and dense coding schemes,
\newblock \emph{Journal of Physics A: Mathematical and General}
\textbf{34}(35), 7081--7094, 2001.
\newblock doi:\href{https://doi.org/10.1088/0305-4470/34/35/332}{10.1088/0305-4470/34/35/332}.

\bibitem{MustoVicary2016}
B. Musto and J. Vicary,
\newblock Quantum Latin squares and unitary error bases,
\newblock \emph{Quantum Information and Computation}
\textbf{16}(15--16), 1318--1332, 2016.
\newblock arXiv:1504.02715.

\bibitem{GoyenecheEtAl2018}
D. Goyeneche, Z. Raissi, S. Di Martino, and K. {\.{Z}}yczkowski,
\newblock Entanglement and quantum combinatorial designs,
\newblock \emph{Physical Review A} \textbf{97}, 062326, 2018.
\newblock doi:\href{https://doi.org/10.1103/PhysRevA.97.062326}{10.1103/PhysRevA.97.062326}.

\bibitem{MustoVicary2019}
B. Musto and J. Vicary,
\newblock Orthogonality for quantum Latin isometry squares,
\newblock \emph{Electronic Proceedings in Theoretical Computer Science}
\textbf{287}, 253--266, 2019.
\newblock doi:\href{https://doi.org/10.4204/EPTCS.287.15}{10.4204/EPTCS.287.15}.

\bibitem{RatherEtAl2022}
S. A. Rather, A. Burchardt, W. Bruzda, G. Rajchel-Mieldzio\'c,
A. Lakshminarayan, and K. {\.{Z}}yczkowski,
\newblock Thirty-six entangled officers of Euler: Quantum solution to a
classically impossible problem,
\newblock \emph{Physical Review Letters} \textbf{128}, 080507, 2022.
\newblock doi:\href{https://doi.org/10.1103/PhysRevLett.128.080507}{10.1103/PhysRevLett.128.080507}.

\bibitem{ZyczkowskiEtAl2023}
K. {\.{Z}}yczkowski, W. Bruzda, G. Rajchel-Mieldzio\'c, A. Burchardt,
S. A. Rather, and A. Lakshminarayan,
\newblock $9\times4=6\times6$: Understanding the quantum solution to Euler's
problem of 36 officers,
\newblock \emph{Journal of Physics: Conference Series} \textbf{2448},
012003, 2023.
\newblock doi:\href{https://doi.org/10.1088/1742-6596/2448/1/012003}{10.1088/1742-6596/2448/1/012003}.

\bibitem{PaczosEtAl2021}
J. Paczos, M. Wierzbi\'nski, G. Rajchel-Mieldzio\'c, A. Burchardt, and
K. {\.{Z}}yczkowski,
\newblock Genuinely quantum solutions of the game Sudoku and their cardinality,
\newblock \emph{Physical Review A} \textbf{104}, 042423, 2021.
\newblock doi:\href{https://doi.org/10.1103/PhysRevA.104.042423}{10.1103/PhysRevA.104.042423}.

\bibitem{TadejZyczkowski2006}
W. Tadej and K. {\.{Z}}yczkowski,
\newblock A concise guide to complex Hadamard matrices,
\newblock \emph{Open Systems \& Information Dynamics} \textbf{13}(2),
133--177, 2006.
\newblock doi:\href{https://doi.org/10.1007/s11080-006-8220-2}{10.1007/s11080-006-8220-2}.

\bibitem{Szollosi2012}
F. Sz\"oll\H{o}si,
\newblock Complex Hadamard matrices of order 6: a four-parameter family,
\newblock \emph{Journal of the London Mathematical Society} \textbf{85}(3),
616--632, 2012.
\newblock doi:\href{https://doi.org/10.1112/jlms/jdr052}{10.1112/jlms/jdr052}; arXiv:1008.0632.

\bibitem{ZhangWangJi2025}
Y. Zhang, X. Wang, and L. Ji,
\newblock Quantum Latin squares with all possible cardinalities,
\newblock \emph{Journal of Combinatorial Designs}, 2026.
\newblock doi:\href{https://doi.org/10.1002/jcd.70021}{10.1002/jcd.70021}; arXiv:2507.05642.

\bibitem{ZangEtAl2025}
Y. Zang, M. Zheng, Z. Tian, and X. Shan,
\newblock On the cardinalities of quantum Latin squares,
\newblock arXiv:2508.01972, 2025.

\bibitem{ZhangCao2026}
Y. Zhang and H. Cao,
\newblock Quantum Latin squares with maximal cardinality,
\newblock \emph{Discrete Mathematics} \textbf{349}, 114863, 2026.
\newblock doi:\href{https://doi.org/10.1016/j.disc.2025.114863}{10.1016/j.disc.2025.114863}.

\bibitem{ZhangJi2026}
Y. Zhang and L. Ji,
\newblock Quantum Latin squares of order $6m$ with all possible cardinalities,
\newblock arXiv:2601.09132, 2026.

\bibitem{Xu2026}
Z. Xu,
\newblock Three quantum Latin squares of order 6 with cardinalities 13, 15,
and 17,
\newblock arXiv:2605.15540, 2026.

\bibitem{ZhangLvCao2026}
Y. Zhang, M. Lv, and H. Cao,
\newblock On the possible cardinalities of quantum Latin squares,
\newblock arXiv:2607.19969v1 [math.CO], 2026.

\end{thebibliography}
\end{document}